\documentclass[12pt,letterpaper,reqno]{amsart}
\usepackage{amssymb}

\usepackage{geometry}
\usepackage{amsmath}
\usepackage{tikz}
\usepackage{caption}
\usepackage{subcaption}
\usetikzlibrary{patterns}
\usepackage{hyperref}
\usepackage{microtype}
\hypersetup{
	pdfstartview={XYZ null null 1.00}, 
	pdfpagemode=UseNone, 
	colorlinks,
	breaklinks, 
	linkcolor=blue,
	urlcolor=blue, 
	anchorcolor=blue,
	citecolor=blue
}
\usepackage[capitalise]{cleveref}
\usepackage{bbold}
\usepackage[shortlabels]{enumitem}

\newtheorem{theorem}{Theorem}[section]
\newtheorem{question}{Question}
\newtheorem{lemma}[theorem]{Lemma}
\newtheorem{prop}[theorem]{Proposition}
\newtheorem{cor}[theorem]{Corollary}

\newtheorem{claim}[theorem]{Claim}

\theoremstyle{definition}
\newtheorem{definition}[theorem]{Definition}
\newtheorem{remark}[theorem]{Remark}

\newtheorem*{acknowledge}{Acknowledgements}

\newcommand{\N}{\mathbb{N}}
\newcommand{\R}{\mathbb{R}}
\newcommand{\F}{\mathbb{F}}
\def\acts{\curvearrowright}
\newcommand{\ep}{\varepsilon}
\newcommand{\w}{\mathfrak{w}}
\newcommand{\coc}{\mathfrak{w}}

\newcommand{\partialto}{\rightharpoonup}
\newcommand{\dom}{\mathrm{dom}}
\newcommand{\im}{\mathrm{im}}
\newcommand{\graph}{\mathrm{graph}}

\newcommand{\set}[1]{\left\{ #1 \right\}}
\newcommand*{\defeq}{\mathrel{\vcenter{\baselineskip0.5ex \lineskiplimit0pt \hbox{\scriptsize.}\hbox{\scriptsize.}}}=}

\DeclareMathOperator{\good}{good}

\newcommand{\eqcomment}[1]{\Big[\text{#1}\Big] \hspace{6pt}}

\makeatletter
\def\l@section{\@tocline{1}{5pt}{0pc}{}{}}
\renewcommand{\tocsection}[3]{%
	\indentlabel{\@ifnotempty{#2}{\makebox[20pt][l]{%
				\ignorespaces#1 #2.\hfill}}}\sc #3\dotfill}

\newdimen{\tocsubsecmarg}
\def\l@subsection{\@tocline{2}{3pt}{0pc}{\tocsubsecmarg}{}}
\renewcommand{\tocsubsection}[3]{%
	\indentlabel{\@ifnotempty{#2}{\makebox[30pt][l]{%
				\ignorespaces#1 #2.\hfill}}}#3\dotfill}
\makeatother
\let\oldtocsubsection=\tocsubsection
\renewcommand{\tocsubsection}[2]{\hspace{3em} \oldtocsubsection{#1}{#2}}

\title{One-ended spanning trees and definable combinatorics}
\author{Matt Bowen}
\author{Antoine Poulin}
\author{Jenna Zomback}
\date{\today}

\begin{document}

\maketitle

\vspace{3mm}

\begin{abstract}
    Let $(X,\tau)$ be a Polish space with Borel probability measure $\mu,$ and $G$ a locally finite one-ended Borel graph on $X.$  We show that $G$ admits a Borel one-ended spanning tree generically.  If $G$ is induced by a free Borel action of an amenable (resp., polynomial growth) group then we show the same result $\mu$-a.e. (resp., everywhere). Our results generalize recent work of Tim\'ar, as well as of Conley,  Gaboriau, Marks, and Tucker-Drob, who proved this in the probability measure preserving setting. We apply our theorem to find Borel orientations in even degree graphs and measurable and Baire measurable perfect matchings in regular bipartite graphs, refining theorems that were previously only known to hold for measure preserving graphs. In particular, we prove that bipartite one-ended $d$-regular Borel graphs admit Baire measurable perfect matchings.
\end{abstract}

\tableofcontents

\pagebreak
\section{Introduction}

In this paper we are concerned with determining which hyperfinite Borel graphs admit Borel one-ended spanning trees and with applications of these objects to definable combinatorics.

Recall that a \textbf{Borel graph} on a Polish space $(X,\tau)$ is a graph whose vertex set $V(G)$ is $X$ and whose edge set $E(G)$ is a Borel subset of $X^2.$  A Borel graph is \textbf{hyperfinite} if it can be written as an increasing union of component finite Borel graphs. A connected graph $G$ is \textbf{one-ended} if for every finite $F\subset V(G)$ the induced subgraph on $V(G)-F$ has exactly one infinite connected component, and a non-connected graph is one-ended if each of its connected components is. Such graphs arise naturally in various contexts, such as actions of amenable groups, and characterizing the combinatorial properties of such graphs has been a major focus of descriptive graph theory since its inception in {the work of Kechris, Pestov. amd Todorcevic} \cite{kst}. See \cite{lovasz,kechris.marks,pikhurko2020borel} for surveys and introductions to this topic. 

In the special case of probability measure preserving (pmp) graphs (also refered to as \textbf{graphings} in the literature), the existence of finitely-ended spanning trees has been well studied and frequently applied.  In the pmp context, a Borel graph is a.e. hyperfinite if and only if it admits a Borel a.e. spanning tree with at most two ends \cite{adams, blps.spanning}. Recently, Tim\'ar \cite{timar} and independently Conley,  Gaboriau, Marks, and Tucker-Drob\cite{cgmtd}, refined this to show that the spanning tree can be chosen to have the same number of ends as the original component.  This latter fact already has a number of applications in probability theory, measurable equidecompositions, and measurable combinatorics; see \cite{bowen.kun.sabok,timar2021nonamenable,timar.new}.

Outside of the pmp setting it is no longer true that hyperfinite graphs admit spanning trees with at most two ends, for example, the Schreier graph of the boundary action of $\F_2$. The so-called tail equivalence relation was shown to be hyperfinite by Dougherty, Jackson, and Kechris in \cite{djk}, and Marquis and Sabok proved hyperfiniteness for boundary actions of hyperbolic groups in \cite{marquis2020hyperfiniteness}.  Since the constructions of one-ended spanning trees in \cite{timar} and \cite{cgmtd} both rely on the existence of at most two-ended spanning trees, the methods from these papers can not be easily adapted to other circumstances.  In the present paper we show that the issue of not having a two-ended spanning tree can typically be overcome. We give more direct constructions of one-ended spanning trees in most of the natural settings where hyperfiniteness is guaranteed.

 Our main result is the following:  

\begin{theorem}\label{trees}
Let $X$ be a Polish space equipped with a Borel probability measure $\mu$. Any locally finite one-ended Borel graph $G$ on $X$ admits a Borel one-ended spanning tree:

\begin{enumerate}[(i)]
    \item \label{intro:comeagre} on an invariant comeagre set.
    \item \label{intro:conull} on an invariant $\mu$-conull set if $G$ is induced by a free Borel action of an amenable group.
    \item \label{intro:poly_growth} everywhere if $G$ is induced by a free Borel action of a polynomial growth group.
\end{enumerate}
\end{theorem}

In fact, in each case of \cref{trees} we not only construct one-ended spanning trees but \textbf{connected toasts} (particular strong witnesses to hyperfiniteness, precisely defined in \cref{toast} and originally introduced by the first author, Kun, and Sabok in \cite{bowen.kun.sabok}), whose existence is potentially stronger; see \cref{toast generic,toast a.e.,toast borel}.

As a consequence of \cref{trees}, we are able to quickly deduce new results on definable matchings and balanced orientations that were previously known only to hold for pmp graphs.  For example, we prove the following Baire measurable analogue of the main result of \cite{bowen.kun.sabok}. Here and in the rest of the paper, by \textbf{generically} we mean on a G-invariant comeagre Borel set.

\begin{theorem}\label{bm match}
Any bipartite one-ended $d$-regular Borel graph has a Borel perfect matching generically.  
\end{theorem}

Previously, the only general classes of bipartite Borel graphs that were known to admit Borel perfect matchings generically were acyclic or of exponential growth (see the work of Conley and Miller \cite[Theorem B]{conley.miller.pm} and Marks and Unger \cite[Theorem 1.3]{marks2016baire} respectively).  Moreover, it is known that neither the one-ended assumption nor the $d$-regular assumption can be dropped from \ref{bm match} due to ergodicity obstructions, even when such graphs admit Borel fractional perfect matchings (see the work of Conley and Kechris \cite[Section 6]{conley.kechris} and \cite[Example 3.1]{bowen.kun.sabok} respectively).

We also show that connected toasts are enough to find \textbf{Borel balanced orientations} (i.e. Borel orientations of the edges of a graph so that the in-degree of each vertex is equal to its outdegree) of even-degree graphs, refining and generalizing \cite[Corollary 3.7]{bowen.kun.sabok}.

\begin{theorem}\label{borel orientation}
Every even-degree Borel graph that admits a connected toast also admits a Borel balanced orientation.  In particular, every one-ended even-degree Borel graph admits a Borel balanced orientation generically, and every free Borel action of a superlinear growth amenable (resp., polynomial growth) group admits a Borel balanced orientation a.e. (resp., everywhere).
\end{theorem}

Again, the irrational rotation graph and its variants show that this result fails for two-ended graphs.  Theorem \ref{borel orientation} also extends some cases of recent results of Thornton, who showed in \cite{thornton2022orienting} that even degree Borel graphs that are pmp (resp., of subexponential growth) admit Borel orientations with outdegree at most $\frac{deg(v)}{2}+1$ a.e. (resp., everywhere), and of Bencs, Hru\v{s}kov\'{a}, and Tóth, who showed in \cite{bht} that even degree planar lattices have factor of i.i.d balanced orientations. 

\subsection{Future work}
\hspace{0.5em}\\

In this paper we were able to show that all locally finite one-ended Borel graphs admit one-ended spanning trees generically, and it was already known that such graphs admit one-ended spanning trees a.e. if they are hyperfinite and pmp.  While we were able to remove the pmp assumption from the latter result for Schreir graphs of actions of amenable groups, it remains unclear if we can do this in general for hyperfinite graphs.

\begin{question}
Does every locally finite, one-ended, hyperfinite Borel graph also admit a Borel one-ended spanning tree a.e.?
\end{question}

We have also shown that the existence of connected toasts is enough to ensure that $d$-regular bipartite Borel graphs admit Borel perfect matchings almost everywhere and generically. It remains unknown if this is enough to have a Borel perfect matching everywhere, but it was shown by Gao, Jackson, Krohne, and Seward in \cite{gjks.2} that the standard Schreir graphs of $\mathbb{Z}^d$-actions admit Borel perfect matchings, so it seems possible that this fact can be extended.

\begin{question}\label{Q2}
Does every $d$-regular bipartite Borel graph that admits a connected toast admit a Borel perfect matching?
\end{question}

Given that the classical K\H{o}nig theorem implies every $d$-regular bipartite graph admits a proper edge coloring with $d$ colors, we could also strengthen \cref{Q2} to ask for a Borel $d$-coloring.  With \cref{bm match} in mind, this is most likely much easier to obtain generically.

\begin{question}
Does every $d$-regular bipartite one-ended Borel graph $G$ admit a Borel edge coloring with $d$ colors generically?
\end{question}

It is known by the work of the first author and Weilacher \cite{bowen2021definable} that every such graph admits a Baire measurable edge coloring with $d+1$ colors even when the one-ended condition is dropped from the above, but as noted already one-endedness is required to find even a single Baire measurable perfect matching (and unlike edge colorings with $d+1$ colors, each color in an edge coloring with $d$ colors is a perfect matching.).


\begin{acknowledge}
The authors thank Anush Tserunyan for pointing out that in the Connes-Feldman-Weiss theorem (see \cite{cfw}), unweighted F\o lner sets suffice in the (weighted) quasi-pmp setting, and suggested looking at Marks's proof in \cite{marks.cfw}.
%
%
We are grateful to Anton Bernshteyn for pointing out that one can use the construction in Theorem 4.8 (a) of \cite{basymd} to produce toasts and not just witnesses to $asi=1$ in actions of polynomial growth groups.
We thank Antoine Gournay
for pointing out that \cref{unif connec} fails for some amenable groups.
Finally, the authors thank their advisors, Marcin Sabok and Anush Tserunyan, for their support and feedback.
\end{acknowledge}

\section{Preliminaries}

Throughout let $(X,\tau)$ be a Polish space with Borel probability measure $\mu$, and let $U_i$ be a countable basis for the open sets in $\tau$. Let $X^{<\omega}$ denote the space of finite subsets of $X$.

Given a graph $G$, we fix the following notation.
\begin{itemize}
    \item $V(G)$ is the set of vertices in $G$.
    \item $E(G)$ is the set of edges in $G$.
    \item For $Y\subseteq V(G)$, $E(Y)$ denotes the set of edges in the induced subgraph of $G$.
    \item $\rho_G: V(G)^2 \rightarrow [0, \infty]$ is the \textbf{graph metric}, where $\rho_G(x,y)=\infty$ means that $x$ and $y$ are in different connected components. 
    \item If $C\subset X$:
\begin{itemize}
    \item $B_n(C) \defeq \left\{x \in X : \rho_G(x, C) \leq n \right\}$ 
    \item
    $\partial_n C \defeq B_n(C) - C$. In particular, we define the (outer) \textbf{boundary} of $C$ to be $\partial C\defeq \partial_1C$. 
    
    \item
    
    The boundary visible from infinity, $\partial_{vis}C,$ is the subset of $\partial C$ so that each vertex in $\partial_{vis}C$ sees an infinite path avoiding $C.$
\end{itemize}
\end{itemize}

\begin{definition}
 If $B \subset V(G) $ and $n$ is a natural number, consider two points related if they are in $B$ and within distance $n$, and take the generated equivalence relation $E_n$.
We say $C$ is an \textbf{$n$-component} of $B$ if it is an equivalence class of $E_n$.
\end{definition}

Given a collection $\mathcal{T}\subseteq X^{<\omega},$ let $\mathcal{T}_1 \subset \mathcal{T}$ denote the family of minimal sets (ordered by containment), $\mathcal{T}_2$ 
denote the family of minimal sets in $\mathcal{T} \setminus \mathcal{T}_1,$ etc.

\subsection{Connected toasts and one-ended spanning trees}
\hspace{0.5em}\\

Here we define connected toasts and show that they can be used to produce one-ended spanning trees.

\begin{definition}\label{toast}
  Given a Borel graph $G$ on X, we say that a Borel collection $\mathcal{T}\subset X^{<\omega}$ is a \textbf{toast} if it 
  satisfies

\begin{enumerate}[(T1)]
    \item \label{toast:exhaustive} {$\bigcup_{K\in \mathcal{T}}E(K)=E(G)$,} and

    \item \label{toast:disjoint} {for every pair $K, L \in \mathcal{T}$ either $B_1(K) \cap L = \emptyset$ or $B_1(K) \subseteq L$, or }$B_1(L) \subseteq K$.


\end{enumerate}

If \labelcref{toast:disjoint} holds but possibly not \labelcref{toast:exhaustive}, we call $\mathcal{T}$ a \textbf{partial toast}.
We say a toast is \textbf{connected} if

\begin{enumerate}[(T3)]
    \item \label{toast:connected}{for every $K \in \mathcal{T}$ the induced subgraph on $K \setminus \bigcup_{K \supsetneq L \in \mathcal{T}} L$ is connected.}
\end{enumerate}

We say a toast is \textbf{layered} if

\begin{enumerate}[(T4)]
    \item \label{toast:layered}for each $i\ge 1$, $\bigcup \mathcal{T}_i\subseteq \bigcup \mathcal{T}_{i+1}$.
\end{enumerate}

We say that $G$ \textbf{admits a toast a.e. (resp., generically)} if there exists a Borel $G$ invariant conull (resp., comeagre) set $X'\subseteq X$ such that there exists a toast for the graph $G|_{X'}$. We define admitting a connected toast, admitting a layered toast, and admitting a layered connected toast a.e. (resp., generically) analogously.
\end{definition}

\begin{prop}\label{toast_implies_tree}
If $G$ admits a connected toast then it admits a Borel one-ended spanning tree.
\end{prop}

\begin{proof}
For each tile $K\in\mathcal{T}$ we may find a tree that spans $K \setminus \bigcup_{K \supsetneq L \in \mathcal{T}} L$ by \labelcref{toast:connected}.  Let $T'$ be the (component finite) spanning forest produced this way.  We may extend this by selecting a vertex $v_K\in \partial K$ with neighbor $u\notin K$ and adding edge $\{u,v_K\}$ to $T'$ for each $K\in\mathcal{T}.$  Let $T$ be the tree produced this way, and note that directing edges in $K$ towards $v_K$ witnesses that $T$ is one-ended.  Further, $T$ is component spanning by \labelcref{toast:exhaustive,toast:disjoint}. 
\end{proof}

\begin{remark}
In the measurable case the reverse implication also holds, i.e. the existence of one-ended spanning trees  ensures the existence of connected toasts a.e.; see \cite{bowen.kun.sabok}.  We do not know if this is the case in the Borel or Baire measurable settings as well.
\end{remark}


\subsection{Full sets}

\begin{definition}
  Call a subset $C \subset G$ \textbf{full} if its complement $C^c$ is connected. 
\end{definition}

\begin{definition}\label{def:filling}
    If a subset $C \subset V(G)$ is not full and its complement has exactly one infinite connected component, we call the \textbf{filling} of $C$, denoted $\widehat{C}$, the union of $C$ and the finite connected components of $C^c$ or equivalently the complement of the unique infinite connected component of $C^c$. 
\end{definition}

\begin{remark}\label{fullremark}
For any set $C$ that can be filled,
\begin{itemize}
    \item $\widehat{C}$ is full.
    \item $C$ contains the inner boundary of $\widehat{C}$. In particular, if $x \in \widehat{C}$ (resp. $\widehat{C}^c$) and $B_n(x) \not\subset \widehat{C}$ (resp. $\widehat{C}^c$), then $B_n(x) \cap C \neq \emptyset$.
\end{itemize}
\end{remark}

\subsection{The Radon--Nikodym cocycle}
\hspace{0.5em}\\

We say that a countable Borel equivalence relation $E$ on $(X,\mu)$ is \textbf{$\mu$-preserving} (resp. \textbf{null-preserving}) if for any partial Borel injection $\gamma : X \partialto X$ with $\graph(\gamma) \subseteq E$, $\mu(\dom(\gamma))=\mu(\im(\gamma))$ (resp. $\dom(\gamma)$ is $\mu$-null if and only if $\im(\gamma)$ is $\mu$-null). 

We note that we may always replace $\mu$ with $\nu\defeq \sum_{m \geq 1} 2^{-m} (\gamma_m)_* \mu$, where $(\gamma_m)_{m \ge 1}$ is a sequence of Borel automorphisms such that $E = \bigcup_{m \geq 1} \text{Graph}(\gamma_m)$, which exists by the Feldman--Moore theorem \cite{feldman.moore}. Now $E$ is $\nu$-null-preserving, and $\mu=\nu$ on every $E$-invariant set. Thus, without loss of generality, assume that $E$ is null-preserving.

By Section 8 in \cite{kechris.miller}, a null-preserving $E$ admits an a.e.~unique \textbf{Radon--Nikodym cocycle} with respect to $\mu$. Being a \textbf{cocycle} for a function $(x,y) \mapsto \coc_x(y) : E \to \R^+$ means that it satisfies the cocycle identity: 
\[
\coc_x(y) \coc_y(z) = \coc_x(z),
\]
for all $E$-equivalent $x,y,z \in X$. With this in mind, we think of $\coc_x(y)$ as the `weight' of $y$ relative to $x$.

We say that $\coc$ is the \textbf{Radon--Nikodym cocycle} of $E$ with respect to $\mu$ if it is Borel (as a real-valued function on the standard Borel space $E$) and for any partial Borel injection $\gamma : X \partialto X$ with $\graph(\gamma) \subseteq E$ and $f \in L^1(X,\mu)$,
\begin{equation}\label{cocycle} 
\int_{\im(\gamma)} f(x)\, d\mu(x) = \int_{\dom(\gamma)} f(\gamma(x))\coc_x(\gamma(x))\, d\mu(x).
\end{equation}

For example, if $E$ is the orbit equivalence relation of a group action $\Gamma\curvearrowright(X,\mu)$ and $\gamma\in\Gamma$, then for all measurable $A\subseteq X$,
\begin{equation} 
\mu(\gamma\cdot A) = \int_{A} \coc_x(\gamma(x))\, d\mu(x).
\end{equation}

We will use the following lemma in the proof of \cref{toast a.e.}. 
Its statement is technical, but this says that if a set $A$ is sufficiently small and we have a Borel family $\mathcal{C}$ satisfying certain boundedness properties, we may discard sets from $\mathcal{C}$ who interfere with $A$ without losing a set of significant measure.

\begin{lemma}\label{lemma:cocycle}
Let $\w:E_G\to \mathbb{R}^+$, $(x,y)\mapsto \w_x(y)$ be the Radon-Nikodym cocycle of $E_G$ with respect to $\mu$.
Let $A$ be a Borel set with $\mu(A)<\varepsilon,$ and $\mathcal{C}\subset X^{<\N}$, $\mathcal{B}\defeq\{B_C\supseteq C:C\in\mathcal{C}\}$ Borel families such that for some finite $B\subseteq G$ with $|B|=k$ and $n\in\N$, for each $x\in B_C\in\mathcal{B}$ and $y\in BB^{-1}x$, $B_C\subseteq B\cdot x_C$ for some $x_C\in C$ and $\w_{y}(x)\leq n$.  Then $$\mu(\bigcup\{C: B_C\cap A\neq \emptyset\})< k^2n\ep.$$
\end{lemma}

\begin{proof}
We will use the fact that if $x,y\in B_C\in\mathcal{B}$, then by assumption, $y=\gamma\cdot x$ for some $\gamma\in BB^{-1}$. We compute:
\begin{align*}
    \mu(\bigcup\{C\in\mathcal{C}: B_C\cap A\neq \emptyset\})
    &= \int \mathbb{1}_{\bigcup\{C\in\mathcal{C}: B_C\cap A\neq \emptyset\}}(x)\;d\mu(x)
    \\
    &\le \int \sum_{y\in B_C:x\in C}\mathbb{1}_{\bigcup\mathcal{B}\cap A}(y)\; d\mu(x) 
    \\
    &= \int \sum_{y\in B_C:x\in C}\frac{\mathbb{1}_{\bigcup\mathcal{B}\cap A}(y)}{\w_x(y)}\w_x(y)\; d\mu(x) \\
    \eqcomment{by \cref{cocycle}} 
    &\le \sum_{\gamma\in BB^{-1}} \int  \mathbb{1}_{\bigcup\mathcal{B}\cap A}(x)\w_{\gamma^{-1}(x)}(x)\; d\mu(x) \\
    &\le |B|^2\cdot n \cdot \mu(A) <k^2n\ep. \qedhere
\end{align*}
\end{proof}

\section{Existence of connected toasts}

\subsection{Connected toasts generically}
\hspace{0.5em}\\

Here we prove that every one-ended bounded degree Borel graph admits a connected toast generically (which implies \cref{intro:comeagre} of \cref{trees} by \cref{toast_implies_tree}). 

\begin{definition}
  Given a sequence of sets $(K_i)_{i\in\omega}$ let $K_{<n}=\bigcup_{i< n}K_{i}.$
\end{definition}

We start with a simple observation:

\begin{lemma}\label{toast.tech}
If $\mathcal{T}$ is a toast and $\mathcal{S}\subseteq \mathcal{T}_{<n}$ is Borel, then $\mathcal{T}\setminus \mathcal{S}$ is a toast. 
\end{lemma}

  



    

\begin{theorem}\label{toast generic}
Any locally finite one-ended Borel graph $G$ on $X$ admits a connected toast generically.
\end{theorem}

\begin{proof}
Such graphs are already known to contain toasts by \cite{btoast} Proposition 6.1.  Let $\mathcal{T}$ be such a toast.  Let $(\gamma_i)_{i\in \omega}$ be a countable set of Borel involutions that generate $E_G,$ and $h:\mathbb{N}^2\rightarrow\mathbb{N}$ a bijection. Recall that we have fixed a countable basis $\set{U_i}$ for $\mathcal{T}$. 

We inductively define Borel partial connected toasts $\mathcal{A}_n\subset \mathcal{T}$ so that:

\begin{enumerate}
    \item $\gamma_i(\bigcup\mathcal{A}_{h(i,j)})$ is non-meagre in $U_j.$
    \item For each $n$, $\mathcal{A}_{<n}$ satisfies \labelcref{toast:disjoint,toast:connected}.
    \item $G\setminus \mathcal{A}_{<n}$ is one-ended and $\mathcal{A}_{<n}\subseteq \mathcal{T}_{<m}$ for some $m$.
   
\end{enumerate} 

To see that we can use this to construct the desired connected toast, note that by property (1) $\gamma_i(\bigcup_n\mathcal{A}_n)$ is comeagre for each $i,$ and thus $A=\bigcap_i\gamma_i(\bigcup_n\mathcal{A}_n)$ is comeagre as well.  Moreover, this set is $G$-invariant as $\bigcup_i\gamma_i(x)\subset A$ and is the $G$ component of $x$ for each $x\in A.$ From this and property (2) of the construction we see that $\{K\in \bigcup_n\mathcal{A}_n: K\subset A\}$ is the desired connected toast.

Now, let $\mathcal{A}_0=\emptyset$ and suppose we have found $\mathcal{A}_n.$  Let $m$ be as in property (3) of the construction and consider $\mathcal{T}'=\mathcal{T}\setminus \mathcal{T}_{<m}.$ This is a toast by Lemma \ref{toast.tech}, and so there exists some $n'$ such that $\gamma_i(\mathcal{T}'_{<h^{-1}(n')})$ is non-meagre in $U_j.$ Notice that for every $K\in \mathcal{T}'$ and $L\in \mathcal{A}_{\leq n},$ either $(L\cup \partial(L))\cap K=\emptyset$ or $L\cup \partial(L)\subset K$ and $K\setminus L$ is connected by our choice of $m.$   

Let $G'=G\setminus \mathcal{A}_{\leq n}.$ This graph is one-ended by property (3) of our construction, and so for every $L\in \mathcal{T}'_{<n'}$ there is an $m'\in\omega$ and $L\subset K\in \mathcal{T}'_{<m'}$ such that $K\setminus L$ is $G'$ connected.  Here, we say $L$ is covered by level $m',$ and observe that for some $m'$ $$\gamma_i(\{x\in L\in \mathcal{T}_{<n'}: \textnormal{ L is covered by level m}\}) $$ is non-meagre in $U_j.$  Applying the Baire category theorem, we may choose $\mathcal{A}_{n+1}$ to be a single covered tile from every maximal $\mathcal{T}'_{<m'}$ tile.  \qedhere

\end{proof}

\begin{remark}\label{hcomb tree}
    We could use a similar algorithm to construct a connected toast (and subsequently a one-ended spanning tree) in any one-ended highly computable graph.  This is outside the scope of the present paper, but it adds to the correspondence between Baire measurable and highly computable combinatorics recently explored by Qian and Weilacher in \cite{qian2022descriptive}.
\end{remark}

\subsection{Connected toasts a.e. in group actions }
\hspace{0.5em}\\

 In this section our goal is to prove that free Borel actions of one-ended amenable groups admit connected toasts a.e..

\begin{theorem}\label{toast a.e.}
Let $\Gamma$ be a finitely generated one-ended amenable group with finite generating set $S,$ and let $G=G(S,a)$ be the Schreier graph of a free null-preserving action $\Gamma\acts^a X.$  Then $G$ admits a layered connected toast a.e.. 
\end{theorem}

Our main combinatorial tool is the construction of F\o lner sequences that witness the group's end structure.

\begin{lemma}\label{connected folner}
Let $\Gamma$ be a finitely generated one-ended amenable group with finite generating set $S$. The Cayley graph of $\Gamma$ admits a sequence of finite sets $F_n\subset B_n$ such that 
\begin{enumerate}[(1)]
    \item\label{Folner1} $\frac{|B_n-F_n|}{|F_n|}\rightarrow 0.$
    \item\label{Folner2} $\partial F_n\subseteq B_n-F_n.$
    \item\label{Folner3} $B_n-F_n$ is connected.
\end{enumerate}
\end{lemma}

{  Note that properties (1) and (2) of the above are just the usual F\o lner condition. Also, it is possible (and from our construction typical) that the sets $F_n$ are not necessarily connected.

Surprisingly, despite Lemma \ref{connected folner} being purely group theoretic in nature, our proof of it relies almost entirely on ideas from descriptive set theory.}

\begin{proof}[Proof of Lemma \ref{connected folner}]
Consider a free pmp action of $\Gamma$ on a standard probability space $(X,\mu).$  The Schreier graph of this action contains a Borel one-ended spanning tree $T$ a.e. by \cite{cgmtd}.  Since $T$ is one-ended we can orient edges toward the unique infinite end in each component, and we say $u$ is \textbf{below} $v$, denoted \textbf{$u<v$}, if there is a directed path from $u$ to $v$.  Say that a vertex $v\in X$ is of \textbf{height} $0$ if it's a leaf in the tree, and then define height(v) := $\max_{u < v}\{ height(u) + 1\}$ (i.e., height $1$ if it's above leaves but not above any other vertices, etc.).  As the tree is a.e. one-ended, a.e. vertex is of finite height.  An $(<)n-tile$ is the set of vertices that are below a given vertex of height $(<)n$ in the tree.

For any given $n$-tile $H$ there is an $m$-tile  $K$ such that $\partial(H)\subseteq K$ for some $m.$  If this is the case, we'll say that $H$ is covered by level $m$ and that $K$ covers $H$.

Now, for fixed $m>n$ and $m$-tile $K$ let $$\good(K)= \{ v\in K: v\in H\textnormal{  is a } <n\textnormal{-tile and } K \textnormal{ covers } H\}.$$  

We will say that $K$ is $(\ep,n)$-good if $\frac{|K\setminus \good(K)|}{|K|}<\ep$.  Notice that any $(\varepsilon,n)$-good tile satisfies the conditions of Lemma \ref{connected folner}, since an $m$-tile minus a $<n$-tile is connected whenever $n<m.$  Therefore, it suffices to show that $(\ep,n)$-good tiles exist for some $n.$  In order to do this, we'll actually show that such tiles almost cover all of the graph.

\vspace{3mm}

\begin{claim}\label{good tiles}
For any $\ep,\ep'>0$ there are $n<m\in\mathbb{N}$ such that $1-\ep'$ of $G'$ is tiled by $(\ep,n)$-good $m$-tiles.  
\end{claim}

\begin{proof}
First note that for any $\varepsilon_2>0$ there is a choice of $n<m$ such that $1-\varepsilon_2$ of $G$ is contained in $<n$-tiles that are covered by level $m$.  Call such an $<n$-tile \textbf{good}.  Fix such an $\varepsilon_2$.  Suppose that $1-\varepsilon_1$ of the graph is tiled by $\ep$-good $m$ tiles.  Then at least $1-\varepsilon_2-(1-\varepsilon_1)=\varepsilon_1-\varepsilon_2$ of the graph is tiled by good $<n$-tiles that are not contained in $\ep$-good $m$ tiles.  Let $S$ be the set of all such $m$-tiles.  Then we know that $\mu(\bigcup_{K\in S}\good(K))\geq \varepsilon_1-\varepsilon_2.$ Moreover, by definition we know that for any $m$-tile $K\in S$ that $|K\setminus \good(K)|\geq \ep|K|\geq \varepsilon|\good(K)|$. So, using the previous two observations and the fact that the action is pmp we see that $\mu(S)\geq \ep(\varepsilon_1-\varepsilon_2)+\varepsilon_1-\varepsilon_2$.  Consequently, we know $$1-\varepsilon_1+(1+\ep)(\varepsilon_1-\varepsilon_2)\leq 1,$$ and thus $$\varepsilon_1\leq \frac{\varepsilon_2(1+\ep)}{\ep}.$$  Therefore, we can choose $n$ and $m$ large enough to make $\varepsilon_2$ small enough to ensure that $\varepsilon_1<\ep'$, as desired.
\end{proof}

\end{proof}

\vspace{3mm}


  
  

\begin{definition}\label{def:iso}
  Let $Y\subseteq X$ be Borel.  We define the \textbf{connected isoperimetric constant} of $G$ restricted to $Y$ to be 
  
   \[
   \Phi^c_Y(G)\defeq \inf_{\mathcal{F}}\set{\frac{\mu(\bigcup_{F\in \mathcal{F}} (B_F-F)\cap Y)}{\mu(\bigcup_{F\in \mathcal{F}}F\cap Y)}},
  \]

Where $\mathcal{F}\subseteq Y^{<\infty}$ is a Borel family with 
$\mu(\bigcup\mathcal{F})>0$ and
associated finite sets $F\subset B_F\subset X$ such that  for each $F\in\mathcal{F}$, 

\begin{enumerate}[(1)]
    \item\label{def:iso1} $\partial_{vis}(F)\subseteq B_F-F$, where $\partial_{vis}$ is with respect to $G$ and not $G|Y.$
    \item\label{def:iso2} $B_F-F$ is connected (as an induced subgraph of $G$).
    \item\label{def:iso3} $B_F\cap F'=\emptyset$ for every $F\neq F'\in \mathcal{F}.$
    \item\label{def:iso4} If $C\subseteq G\setminus \bigcup_{F\in \mathcal{F}}F $ is a finite induced component, then $C\subseteq \bigcup_{F\in \mathcal{F}}B_F.$
\end{enumerate}

\end{definition}

\begin{remark}
When $Y=X,$ $G$ is one-ended, and $\mathcal{F}$ is a family satisfying the above conditions, then the induced subgraph on $X\setminus \bigcup \mathcal{F}$ is a union of one-ended and finite components. Namely, any $G$ path with endpoints in $G\setminus \bigcup \mathcal{F}$ passing through some $F\in \mathcal{F}$ may be rerouted through $B_F.$ 
\end{remark}

The following Lemma will be needed for a measure exhaustion later. 

\begin{lemma}\label{exhaustion}
Let $\mathcal{F}$ witness that $\Phi_X^c(G)<\varepsilon$ and $Y=X\setminus\bigcup_{F\in\mathcal{F}}B_F.$  If $\mathcal{F}'$ witnesses that $\Phi_Y^c(G)<\epsilon',$ then $\mathcal{F}\cup \mathcal{F}'$ witnesses that $\Phi_X^c(G)<\varepsilon+\varepsilon'.$
\end{lemma}

\begin{proof}
The only potential problem is that $B_{F'}\cap F\neq \emptyset$ for some $F'\in\mathcal{F}'$ and $F\in \mathcal{F}.$  However, this can be overcome by replacing $B_{F'}$ with $(B_{F'}-\hat{F})\cup B_F$. To see that this is connected, note that property (4) of the definition and the fact that $F'\subset Y$ ensures that any $G$ path starting in $F'$ intersecting $F$ must first pass through $\partial_{vis}F\subseteq B_F$. In particular, any such path with endpoints in $\partial_{vis}F'$ can be rerouted through $B_F,$ witnessing that this set is connected. Since $B_{F'}$ is finite, it can only intersect finitely many sets $F\in \mathcal{F}$, and so iterating this for each such $F'$ gives the desired $B_{F'}.$
\end{proof}

\begin{remark}
 In the previous definition and lemma it is possible that some sets are nesting, i.e. there may be sets $F,F'\in \mathcal{F}$ so that any infinite walk starting from $\partial F$ must pass through $F'.$  In such a situation it must also be the case that $B_F$ is in a finite component $C$ of $G-F',$ and in particular $C-\hat{F}$ is connected, where $\hat{F}$ is the filling of $F$ as in \cref{def:filling}. Namely, $C$ and $\hat{F}$ satisfy property 3 of being a connected toast.
\end{remark}


Our next proposition is similar to the main Lemma from Marks' proof of the Connes-Feldman-Weiss theorem \cite{marks.cfw}, with the objects from Lemma \ref{connected folner} taking the place of arbitrary F\o lner sequences.

\begin{prop}\label{prop:isoperimetric}
For $G$ as above and any positively measured Borel $H\subseteq X$, $\Phi^c_{H}(G|_H)=0$.
\end{prop}

\begin{proof}
Let $(F_n)_{n\in\omega}$ be a F\o lner sequence as in Proposition \cref{connected folner}, and $n$ be large enough so that $|B_n-F_n|<\varepsilon |F_n|.$  Then we see that $$\int_H |(B_n-F_n)^{-1}x|d\mu(x)<\int_H \varepsilon|F_n^{-1}x|d\mu(x), $$ (where $|(B_n-F_n)^{-1}x|$ and $|F_n^{-1}x|$ are calculated with respect to $G$ and not to $H$).  Since the sets $B_n$ and $F_n$ are finite and bounded, we can find a finite Borel coloring $c$ of $G$ so that $x$ and $y$ get different colors if $B_nx\cap B_ny\neq \emptyset.$  Since every element of $G$ is colored, we see that 

$$\sum_k\int_H|\{z\in (B_n-F_n)^{-1}x: c(z)=k\}|d\mu(x)<\varepsilon \sum_k\int_H|\{z\in F_n^{-1}x: c(z)=k\}|d\mu(x).$$

There is some $k$ that still satisfies this inequality, i.e. $$\int_H|\{z\in (B_n-F_n)^{-1}x: c(z)=k\}|d\mu(x)<\varepsilon \int_H|\{z\in F_n^{-1}x: c(z)=k\}|d\mu(x).$$

Notice that $|\{z\in F_n^{-1}x: c(z)=k\}|$ is $1$ if $x\in F_nz$ for some $z$ with $c(z)=k$ and $0$ otherwise by our choice of coloring.  This means the right hand side of the above equation is just $\mu(\bigcup_{z: c(z)= k}F_nz\cap H).$  Similarly, $|\{z\in (B_n-F_n)^{-1}x: c(z)=k\}|$ is $1$ if $x\in (B_n-F_n)z$ for some $z$ with $c(z)=k$ and $0$ otherwise, so the left hand side of the equation is just $\mu(\bigcup_{z:c(z)=k}(B_n-F_n)z\cap H).$ \cref{Folner2,Folner3} of \cref{connected folner} ensure that the tiles $F_nz$ for $z$ with $c(z)=k$ are as in \cref{def:iso1,def:iso2} of \cref{def:iso}. 
\end{proof}


We now use \cref{prop:isoperimetric,lemma:cocycle} to prove \cref{toast a.e.}. This proof is similar to the equivalence of hyperfiniteness and isoperimetric constant on any subset being equal to zero due to Kaimanovich in \cite{Kaimanovich}. We use the connected isoperimetric constant to ensure that our toast is connected.

\begin{figure}\label{fig:C_k_to_T_k}
    \centering
    \input{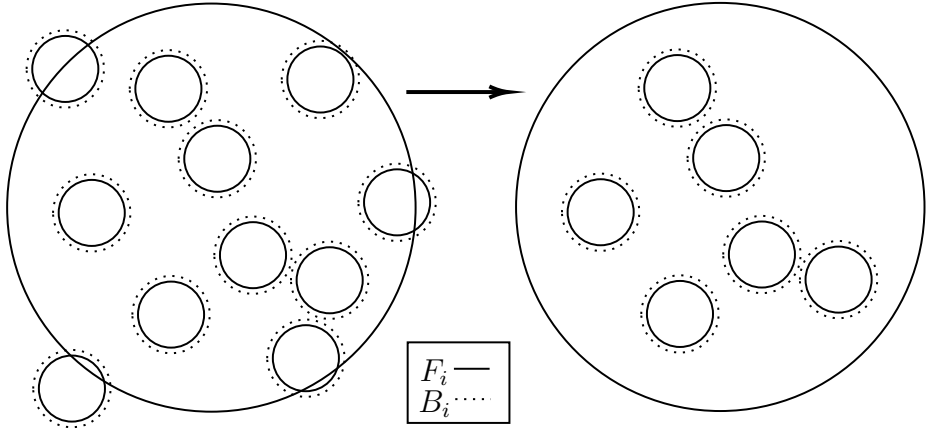}
    \caption{from $\mathcal{C}_k$ to $\mathcal{T}_k$: we only discard a set of measure $\frac{1}{2^{k+1}}$ by \cref{lemma:cocycle}}
    \label{fig:a.e.}
\end{figure}

\begin{proof}[Proof of \cref{toast a.e.}]
By a standard measure exhaustion argument and \cref{exhaustion}, for any $\ep>0$, we can find a maximal countable collection $\mathcal{A}_\ep=\set{A_n}$ of positively measured Borel subsets of $X$ such that:

\begin{enumerate}
    \item Each $A_n$ is of the form $\bigcup \mathcal{F}_n,$ where $\mathcal{F}_n$ is a set witnessing that $\Phi_X^c(G)<\ep$ as in Definition \ref{def:iso}. 
    \item For $A_n\neq A_m\in \mathcal{A}_\ep,$ if $F_n\in \mathcal{F}_n$ and $F_m\in\mathcal{F}_m,$ then $B_{F_n}\cap F_m=\emptyset.$
\end{enumerate}

Given such an $\mathcal{A}_\ep$ let $\mathcal{B}_\ep=\bigcup_n\bigcup_{F\in \mathcal{F}n}B_F,$ where $B_F$ is as in Definition \ref{def:iso}.  Notice that $\mu(\bigcup\mathcal{B})=1;$ if not, then we could apply Proposition \ref{prop:isoperimetric} to $X-\bigcup \mathcal{B}_\ep$ to find another disjoint Borel set $\mathcal{A}_\omega$ as above, contradicting the maximality of $\mathcal{A}_\ep.$  In particular, we have $\mu(\bigcup{\mathcal{A}_\ep})=1-\mu(\bigcup(\mathcal{B}_\ep-\mathcal{A}_\ep))> 1-\ep.$

We will use this fact to inductively build the connected toast as follows.  First, fix such an $\mathcal{A}_{\frac{1}{8}}$. We may refine $\mathcal{A}_{\frac{1}{8}}$ to a subset $\mathcal{C}_0$ so that properties (1)-(4) of Definition \ref{def:iso} still hold, $\mu(\mathcal{C}_0)>\frac{1}{4}$, and there is a finite $B_0\subseteq G$ ($|B_0|=k_0$) and $l_0\in \N$ such that for each $C\in\mathcal{C}_0$, $B_C\subseteq B_0\cdot x_C$ for some $x_C\in C$, and $\mathfrak{w}_y(x)\leq l_0$ for each $x\in 
B_C$ and $y\in B_0B_0^{-1}x$. 

Now assume we have constructed $\mathcal{C}_n\subseteq \mathcal{A}_{\ep'}$ so that properties (1)-(4) of Definition \ref{def:iso} still hold, $\mu(\mathcal{C}_n)>1-\frac{1}{2^{n+2}}$, and so that for $B\in\mathcal{B}_\ep'$ we have $B\subseteq B_n\cdot x'$, $|B_n|=k_n$, and $\max_{x,y\in B}\mathfrak{w}_y(x)<l_n$ for some $k_n,l_n\in\mathbb{N}$.

Let $\ep<\frac{1}{2^{n+2}\cdot k_n^2\cdot l_n}$. Then by the above argument, we can find an $\mathcal{A}_{\ep}$. Let $\mathcal{T}_n\defeq\mathcal{C}_n\setminus\bigcup\set{C\in\mathcal{C}_n:B_C\cap\mathcal{A}^c_\ep\neq\emptyset}$ (see \cref{fig:a.e.}). Since $\mu(\mathcal{A}_{\ep})<\ep$, by \cref{lemma:cocycle}, \[\mu(\mathcal{T}_k)>1-\frac{1}{2^{k+1}}.
\]
Then $\mathcal{T}\defeq\bigcup_{k\in\N}\mathcal{T}_k$ is a layered connected toast a.e.. Indeed, \labelcref{toast:exhaustive} follows from the fact that elements of $\mathcal{T}_n$ are induced components and $\mu(\bigcup\mathcal{T}_n)\rightarrow 1.$ \labelcref{toast:disjoint} follows from construction of the $\mathcal{A}_\ep$ and the fact that we discard sets in $\mathcal{C}_n$ whose connected boundaries interfere with $\mathcal{C}_{n+1}^c$. \labelcref{toast:connected} follows from the fact that $B_F-F$ is connected for each element of a $\mathcal{C}_n$ Finally, \labelcref{toast:layered} follows from the construction, since we throw away sets in $\mathcal{C}_k$ who lie in the complement of $\bigcup \mathcal{C}_{k+1}$.
\end{proof}

\subsection{Connected toasts and polynomial growth groups}\label{connectoastborel}
\hspace{0.5em}\\

We start with a proposition establishing a uniform constant for the boundary of connected sets with connected complement in one-ended groups.

\begin{prop}\label{unif connec}
Let $\Gamma$ be a one-ended, finitely presented group with finite generating set $S$ and consider the Cayley graph metric on $\Gamma$ induced by $S$. There is a constant $\kappa_\Gamma$ such that if $C \subset \Gamma$ is a connected subset of $\Gamma$ with $C^c$ connected, $\partial_{\kappa_\Gamma} C$ is connected.
\end{prop}

A proof of this result for balls can be found in \cite{Gournay_2014} for spheres, but as noted (in \cite{Gournay_2014}) the construction works for all sets satisfying the hypothesis. See also \cite{Gournay2}. We give a sketch of the proof for a slightly stronger result, \cref{fullcor}.

\begin{lemma}\label{fullcor}
Let $\Gamma$ be a one-ended, finitely presented group with finite generating set $S$ and consider the Cayley graph metric on $\Gamma$ induced by $S$. There is a constant $\kappa_\Gamma$ such that if $C \subset \Gamma$ is a full subset of $\Gamma$ with $B_{\kappa_\Gamma}(C)$ connected, then $\partial_{\kappa_\Gamma} C$ is connected.
\end{lemma}

\begin{remark}
The constant $\kappa_\Gamma$ in Lemma \ref{fullcor} can be taken to be the same as in Proposition \ref{unif connec}.
\end{remark}

\begin{figure}

\captionsetup[subfigure]{font=footnotesize}
\centering
\subcaptionbox{Possibility for the van Kampen diagram to not be homeomorphic to a disc.}[.5\textwidth]{%
\tikzset{every picture/.style={line width=0.75pt}} 

\begin{tikzpicture}[x=0.75pt,y=0.75pt,yscale=-1,xscale=1]

\draw    (270.8,198.79) -- (338,155.53) -- (338,107.33) -- (298,80.14) -- (272.4,50.47) -- (309.2,33.17) -- (390.8,55.41) -- (389.2,80.14) -- (338,107.33) ;
\draw    (338,155.53) -- (379.6,196.32) ;
\draw    (298,80.14) -- (350,44.29) ;
\draw    (324,62.21) -- (389.2,80.14) ;
\draw    (304.4,177.16) -- (368.4,185.2) ;
\draw  [dash pattern={on 0.84pt off 2.51pt}]  (379.6,196.32) -- (402,217.33) ;
\draw  [dash pattern={on 0.84pt off 2.51pt}]  (270.8,198.79) -- (250,216.1) ;
\draw   (263,86) -- (298,80.14) -- (302.6,130.66) -- (277.6,143.66) -- (259,118) -- cycle ;
\draw    (259,118) -- (302.6,130.66) ;
\draw    (263,86) -- (280.8,124.33) ;

\end{tikzpicture}}%
\subcaptionbox{Removing the striked region will cause a disconnect in the van Kampen diagram.}[.5\textwidth]{
 
\tikzset{
pattern size/.store in=\mcSize, 
pattern size = 5pt,
pattern thickness/.store in=\mcThickness, 
pattern thickness = 0.3pt,
pattern radius/.store in=\mcRadius, 
pattern radius = 1pt}
\makeatletter
\pgfutil@ifundefined{pgf@pattern@name@_5z1svear8}{
\pgfdeclarepatternformonly[\mcThickness,\mcSize]{_5z1svear8}
{\pgfqpoint{0pt}{0pt}}
{\pgfpoint{\mcSize+\mcThickness}{\mcSize+\mcThickness}}
{\pgfpoint{\mcSize}{\mcSize}}
{
\pgfsetcolor{\tikz@pattern@color}
\pgfsetlinewidth{\mcThickness}
\pgfpathmoveto{\pgfqpoint{0pt}{0pt}}
\pgfpathlineto{\pgfpoint{\mcSize+\mcThickness}{\mcSize+\mcThickness}}
\pgfusepath{stroke}
}}
\makeatother
\tikzset{every picture/.style={line width=0.75pt}} 

\begin{tikzpicture}[x=0.75pt,y=0.75pt,yscale=-1,xscale=1]

\draw    (248.27,231.6) -- (257.04,224.38) ;
\draw    (257.04,224.38) -- (265.8,217.15) ;
\draw    (265.8,217.15) -- (274.56,209.93) ;
\draw    (274.56,209.93) -- (283.33,202.71) ;
\draw    (283.33,202.71) -- (292.09,195.49) ;
\draw    (292.09,195.49) -- (300.85,188.27) ;
\draw    (300.85,188.27) -- (275.82,162.27) -- (300.85,139.99) -- (268.31,91.71) -- (288.34,70.66) -- (335.91,53.33) -- (360.95,97.9) -- (343.42,142.47) -- (300.85,139.99) ;
\draw    (300.85,188.27) -- (333.4,229.12) ;
\draw    (268.31,91.71) -- (348.43,75.62) ;
\draw    (294.6,130.09) -- (348.43,75.62) ;
\draw    (257.04,224.38) -- (324.64,217.98) ;
\draw    (378.47,178.13) -- (396,199.17) ;
\draw  [dash pattern={on 0.84pt off 2.51pt}]  (248.27,231.6) -- (232,248.93) ;
\draw  [pattern=_5z1svear8,pattern size=6pt,pattern thickness=0.75pt,pattern radius=0pt, pattern color={rgb, 255:red, 0; green, 0; blue, 0}] (300.85,139.99) -- (343.42,142.47) -- (378.47,177.13) -- (300.85,188.27) -- (275.82,162.27) -- cycle ;
\draw    (378.47,177.13) -- (324.64,217.98) ;
\draw  [dash pattern={on 0.84pt off 2.51pt}]  (333.4,229.12) -- (350.93,250.17) ;
\draw  [dash pattern={on 0.84pt off 2.51pt}]  (396,199.17) -- (413.53,220.22) ;

\end{tikzpicture}}
\caption{Two possible case pictured for proof generality in the proof of Lemma \ref{fullcor}.}
\label{fig:vKD}
\end{figure}
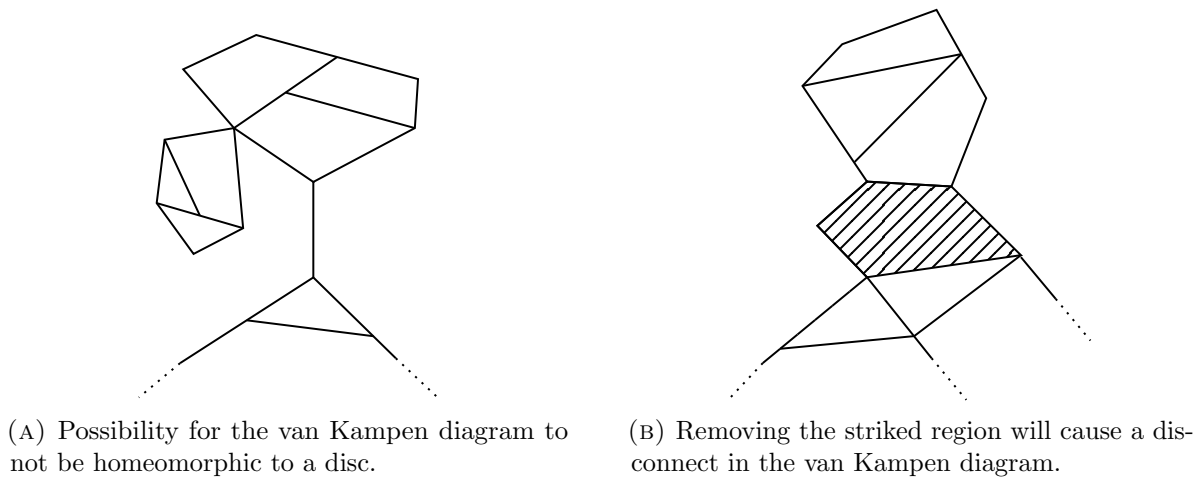

\begin{proof}[Sketch of proof]
The proof is almost the same as of Proposition \ref{unif connec} in \cite{Gournay_2014}. Take two points in $\partial_{\kappa_\Gamma}(V)$. Take one path avoiding $V$, say $p_1$ and one path within $B_{\kappa_\Gamma}(V)$, say $p_2$. This determines a loop, namely $p_1p_2^{-1}$, in the Cayley complex of $\Gamma$, which is simply connected. Thus, one can fill this loop by the image of a disc which can be decomposed into $2$-cells corresponding to generating relations, of which there are only finitely many of. This decomposition is commonly referred to as a van Kampen diagram, which we denote $D$. It will not always be homeomorphic to a disc, since there might be some degeneracies, in the form of flattening part of the disc to a line or point. This is pictured on the left in Figure \ref{fig:vKD}.

\begin{figure}[h]
    \centering

    \tikzset{every picture/.style={line width=0.75pt}} 

\begin{tikzpicture}[x=0.75pt,y=0.75pt,yscale=-1,xscale=0.87]

\draw  [fill={rgb, 255:red, 0; green, 0; blue, 0 }  ,fill opacity=1 ] (62.47,129.66) .. controls (62.47,128.38) and (61.28,127.34) .. (59.81,127.34) .. controls (58.34,127.34) and (57.15,128.38) .. (57.15,129.66) .. controls (57.15,130.95) and (58.34,131.99) .. (59.81,131.99) .. controls (61.28,131.99) and (62.47,130.95) .. (62.47,129.66) -- cycle ;
\draw  [fill={rgb, 255:red, 0; green, 0; blue, 0 }  ,fill opacity=1 ] (586.02,130.83) .. controls (586.02,129.54) and (584.83,128.5) .. (583.37,128.5) .. controls (581.9,128.5) and (580.71,129.54) .. (580.71,130.83) .. controls (580.71,132.11) and (581.9,133.15) .. (583.37,133.15) .. controls (584.83,133.15) and (586.02,132.11) .. (586.02,130.83) -- cycle ;
\draw  [color={rgb, 255:red, 23; green, 11; blue, 11 }  ,draw opacity=1 ][dash pattern={on 4.5pt off 4.5pt}] (48,183.6) .. controls (48,172.78) and (56.78,164) .. (67.6,164) -- (562.4,164) .. controls (573.22,164) and (582,172.78) .. (582,183.6) -- (582,242.4) .. controls (582,253.22) and (573.22,262) .. (562.4,262) -- (67.6,262) .. controls (56.78,262) and (48,253.22) .. (48,242.4) -- cycle ;
\draw  [color={rgb, 255:red, 0; green, 0; blue, 0 }  ,draw opacity=1 ][dash pattern={on 4.5pt off 4.5pt}] (40,133.2) .. controls (40,109.89) and (58.89,91) .. (82.2,91) -- (558.8,91) .. controls (582.11,91) and (601,109.89) .. (601,133.2) -- (601,259.8) .. controls (601,283.11) and (582.11,302) .. (558.8,302) -- (82.2,302) .. controls (58.89,302) and (40,283.11) .. (40,259.8) -- cycle ;
\draw    (150.17,142.05) -- (126.25,193.15) ;
\draw    (174.09,53.8) -- (123.59,128.12) ;
\draw    (146.18,152.5) -- (224.58,216.37) ;
\draw    (148.84,90.96) -- (217.94,135.08) ;
\draw    (217.94,135.08) -- (259.13,175.73) ;
\draw    (259.13,175.73) -- (224.58,216.37) ;
\draw    (259.13,39.86) -- (195.35,121.15) ;
\draw    (235.21,70.06) -- (312.29,117.67) ;
\draw    (313.62,37.54) -- (312.29,117.67) ;
\draw    (275.41,136.54) -- (328.23,179.21) ;
\draw    (259.13,175.73) -- (340.19,223.34) ;
\draw    (328.23,179.21) -- (309.63,204.76) ;
\draw    (421.25,44.51) -- (388.03,117.67) ;
\draw    (388.03,117.67) -- (495.66,200.11) ;
\draw    (328.23,179.21) -- (534.2,78.18) ;
\draw    (312,110) -- (407.96,73.54) ;
\draw    (335,102) -- (388.03,117.67) ;
\draw    (407.96,73.54) -- (493.01,97.92) ;
\draw    (493.01,97.92) -- (575,156) ;
\draw    (441.85,158.89) -- (368.1,222.18) ;
\draw    (441.85,158.89) -- (546.16,136.25) ;
\draw    (534.2,78.18) -- (546.16,136.25) ;
\draw   (128,70) -- (206,42) -- (386,34) -- (550,83) -- (585,132) -- (565,175) -- (467,210) -- (355,224) -- (262,221) -- (146,208) -- (70,150) -- (59,129) -- (74,104) -- cycle ;
\draw [line width=3]    (585,132) -- (583.33,135.83) ;
\draw [line width=3]    (583.33,135.83) -- (581.67,139.67) ;
\draw [line width=3]    (581.67,139.67) -- (580,143.5) ;
\draw [line width=3]    (580,143.5) -- (578.33,147.33) ;
\draw [line width=3]    (578.33,147.33) -- (576.67,151.17) ;
\draw [line width=3]    (576.67,151.17) -- (575,155) ;
\draw [line width=3]    (575,155) -- (492,97) -- (413,137) -- (388,118) -- (334,103) -- (312,111) -- (313,118) -- (278,139) -- (217,135) -- (149,141) -- (124,128) -- (72,105) -- (58,131) ;

\draw (556.2,235.5) node [anchor=north west][inner sep=0.75pt]   [align=left] {$\displaystyle \textcolor[rgb]{0.05,0.05,0.05}{V}$};
\draw (523.86,274.1) node [anchor=north west][inner sep=0.75pt]  [color={rgb, 255:red, 0; green, 0; blue, 0 }  ,opacity=1 ] [align=left] {$\displaystyle B_{\kappa _{\Gamma }}( V){}$};
\draw (315.57,13.64) node [anchor=north west][inner sep=0.75pt]   [align=left] {$\displaystyle p_{1}$};
\draw (309.93,230.59) node [anchor=north west][inner sep=0.75pt]   [align=left] {$\displaystyle p_{2}$};

\end{tikzpicture}
    
    \caption{Example of a van Kampen diagram, without the possible generalities of Figure \ref{fig:vKD}. The bold line is the constructed path, after removing every cell intersecting $B_{\kappa_\Gamma}(V)^c$, as per the procedure described in Lemma \ref{fullcor}.}
    
    \label{fig:vK}
\end{figure}
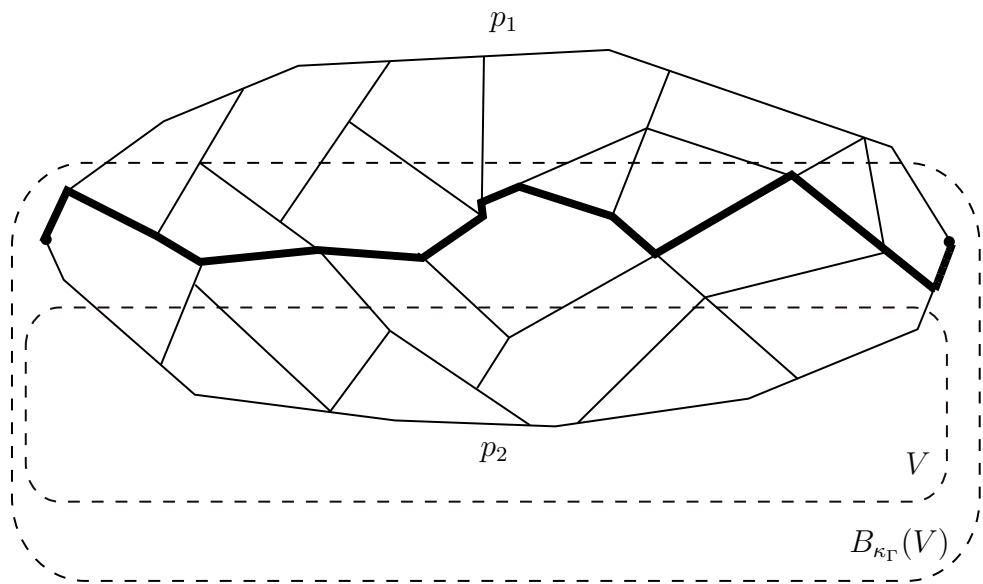

Such $2$-cells can have vertices on the outside of $B_{\kappa_\Gamma}(V)$. Pick such a $2$-cell which intersect the boundary in $p_1$ and remove the cell and its intersection with the boundary of $D$. If there does not exists such a $2$-cell, $p_1$ is already a path in $\partial_{\kappa_\Gamma}(V)$. 

Upon removal, the boundary changes and $p_1$ changes to be a new path $p'$. More precisely, the new boundary of $D$ is of the form $p'p_2^{-1}$ for some path $p'$. It is possible that removing a $2$-cell disconnects the disc, as seen on the right in Figure \ref{fig:vKD}. In which case, $p_2$ is still part of the boundary of exactly one connected component, which is the one on which the argument should be continued.

Note that by taking $\kappa_\Gamma$ to be bigger than half of the longest generating relations, we ensure that a given $2$-cell, which corresponds to a generating relation, cannot have both vertices outside of $B_{\kappa_\Gamma}(V)$ and vertices within $V$. In particular, the new path $p'$ cannot intersect $V$.

The remaining path after removing all $2$-cells with vertices in $B_{\kappa_\Gamma}(V)^c$ can never fall within $V$ either, since for this to happen, a $2$-cell with vertices in $V$ would have had to be removed. See Figure \ref{fig:vK}.

We thus end with a path between the two points which lies in $\partial_{\kappa_\Gamma}(V)$. Since the points were arbitrary, $\partial_{\kappa_\Gamma}(V)$ must be connected, as required. 
\end{proof}

\begin{theorem}\label{toast borel}
Let $\Gamma$ be a one-ended polynomial growth group with finite generating set $S,$ and let $G = G(S,a)$ be the Schreier graph of a free Borel action $\Gamma\acts^a X$.  Then $G$ admits a connected toast (everywhere).
\end{theorem}

\begin{proof}
Since $\Gamma$ is of polynomial growth, it must be finitely presented. Let $\kappa_\Gamma$ be the constant given by Proposition \ref{unif connec}, which is also sufficient for Lemma \ref{fullcor}. 

By \cite{basymd}, this action has finite Borel asymptotic dimension, hence by the proof of Theorem 4.8 a), for $r >> \kappa_\Gamma$, there exists a sequence of sets $V_n \subset X$ such that:
\begin{enumerate}[(P1)]
    \item\label{pretoast:finite} The 3r-components of $V_n$ are uniformly bounded by some constant $c_n$, in the sense that they are of diameter $< c_n$. This also yields a uniform bound on their cardinality.
    \item\label{pretoast:exhaust} For every $x \in X$, there is $n$ such that $B_r(x) \subset V_n$.
    \item\label{pretoast:disjoint} If $n < m$ and $x \in V_n$, either $B_{4r}(x) \subset V_m$ or $B_{4r}(x) \cap V_m = \emptyset$.
\end{enumerate}
Uniform boundedness in \labelcref{pretoast:finite} is not direct. It follows from the fact that these sets are constructed using \cite{basymd} Corollary 4.6, which gives a finite cover $\mathcal{U}^n = \left\{U_0^n, ..., U_d^n \right\}$ where every $U^n_i$ has uniformly bounded $3r$-components. The sequence $V_n$ is built locally the same as a member of the cover, in the sense that for any $x$, there is some $i$ such that
\begin{align*}
    V_n|_{[x]} = U_i^n|_{[x]}.
\end{align*}
Thus, the $3r$-components of $V_n$ are uniformly bounded by some $c_n$.

 As noted in \cite{basymd}, these are similar to toasts, but lack some properties. We modify them to be toasts by considering the sequence $W_n := \widehat{V_n}$.
 
 To see why this is well-defined, we just need to justify why $V_n$ can be filled. Consider that the $3r$-components are far apart from each other, with respect to $\kappa_\Gamma$, filling $V_n$ will be equivalent to filling every $3r$-component separately.
 The only problem that might happen is if two $3r$-components $C_1, C_2$ are nested in the sense that $C_1$ is strictly contained in a finite connected component of $C_2^c$. Since every component is isomorphic to the Cayley graph and the $3r$-components have bounded diameter, there can be no infinitely nested sequence of $3r$-components and so we can restrain ourselves to looking at $3r$-components which are nested in no other $3r$-components.

We now prove \labelcref{toast:exhaustive} to \labelcref{toast:connected} for the collection $\mathcal{T} \subset X^{<\omega}$ of connected components (with respect to the Schreier graph) of the sequence $W_n$, which are finite because the $3r$-components are bounded. More precisely, since $3r$-components of $W_n$ are fillings of $3r$-components of $V_n$ and these are bounded in diameter by some constant $c_n$, the corresponding $3r$-components of $W_n$ will still be bounded in diameter by $c_n$. By \labelcref{pretoast:exhaust}, any edge of $G$ is eventually in $E(V_n)$, hence in $E(W_n)$ and $\bigcup_{L\in \mathcal{T}}E(L)$. \begin{flushright} \vspace{-1.5em} $\square$ \labelcref{toast:exhaustive} \end{flushright}

To prove \labelcref{toast:disjoint}, suppose we have $L_1\neq L_2 \in \mathcal{T}$ with $L_i$ a connected component of $W_{n_i}$. If $n_1 = n_2$, 
then $L_1$ and $L_2$ are different connected components of the same set, hence $B_1(L_1) \cap L_2 = \emptyset$. When $n_1 < n_2$, it suffices to show one of

\begin{enumerate}
    \item $B_r(L_1) \cap L_2 = \emptyset$
    \item $B_r(L_2) \subset L_1$
    \item $B_r(L_1) \subset L_2$.
\end{enumerate}

Note that we are considering the ball of radius $r$, not just one. Considering that both $\partial_r(L_i)$ are finite and connected, if they are disjoint, either $\partial_r(L_1) \subset L_2$, which implies (3), or $\partial_r(L_1) \subset B_r(L_2)^c$.  By symmetry, let us consider the case where $\partial_r(L_1) \subset B_r(L_2)^c$ and $\partial_r(L_2) \subset B_r(L_1)^c$. In this case, (1) is true.

Consider now the case where the $\partial_r(L_i)$ are not disjoint. Let $y \in \partial_r(L_1) \cap \partial_r(L_2)$ and by Remark \ref{fullremark}, we can find two points $x_i \in L_i \cap V_{n_i}$ which are $r$-close to $y$. In this case, by \labelcref{pretoast:disjoint}, $B_{4r}(x_1) \subset V_{n_2}$ and by connectedness, $B_{4r}(x_1) \subset L_2$. In particular, $y \in L_2$, contradicting the hypothesis that $y \in \partial_r(L_2)$.  \begin{flushright}\vspace{-1.5em} $\square$ \labelcref{toast:disjoint} \end{flushright} 

To prove \labelcref{toast:connected}, note that we've shown something stronger than \labelcref{toast:disjoint}. We've shown that if $L_1, L_2 \in \mathcal{T}$ and $L_1 \subset L_2$, we have that $\partial_{\kappa_\Gamma}(L_1) \subset L_2$. Furthermore, $\partial_{\kappa_\Gamma}(L_1)$ is either disjoint of any other $L_3 \in \mathcal{T}$ with $L_3 \subset L_2$, or such $L_3$ is a connected component of the same $W_n$ that $L_1$ is.

Taking all such $L_3$ repeatedly, we get a subset of $W_n$ whose $\kappa_\Gamma$-ball is connected and contained in $L_2$ without touching any other components. The $\kappa_\Gamma$-boundary is connected by the Lemma \ref{fullcor}. \begin{flushright} \vspace{-1.5em} $\square$ \labelcref{toast:connected} \end{flushright}

Hence, the connected components of the sequence $W_{n}$ is a connected toast, as desired.
\end{proof}

\section{Combinatorial applications}

In this section we use the connected toasts constructed in the previous sections to help produce measurable and Baire measurable matchings in regular bipartite graphs and Borel balanced orientations in even degree graphs.

\subsection{Borel orientations}
\hspace{0.5em}\\

 In this section we prove Theorem \ref{borel orientation}.  The same result was already proven for pmp graphs in \cite{bowen.kun.sabok} as an application of the results on measurable perfect matchings in that paper.  Unlike in the measure case, in the Borel setting it is no longer known if the existence of connected toasts ensures the existence of Borel matchings even in $d$-regular bipartite graphs.  Instead of using the seemingly more complicated matching results, we give a more direct proof that even degree graphs with connected toasts admit Borel balanced orientations. 

We use the following simple fact, whose proof can be found in \cite{bowen.kun.sabok} Lemma 2.4.

\begin{lemma}\label{subgraph}
Let $G$ be a finite, connected graph and $P \subseteq V(G)$ a subset of even size.
Then there exists a spanning subgraph $H$ of $G$ such that every vertex of $P$ has odd degree in $H$, and every vertex of  $V(G) \setminus P$ has even degree in $H$.
\end{lemma}

\begin{theorem}\label{borel or}
Let $G$ be a locally finite Borel graph that admits a Borel connected toast.  If each vertex of $G$ has even degree then $G$ admits a Borel balanced orientation.  
\end{theorem}

\begin{proof} 
Let $\mathcal{T}$ be a Borel connected toast. We inductively build a sequence of even degree one-ended Borel graphs $G=G_0\supset G_1\supset G_2...$ such that $E(G_n)\setminus E(\mathcal{T}_n)=E(G_{n+1})\setminus E(\mathcal{T}_n)$ as follows:

Suppose we have defined $G_n.$  We claim that for every edge $e\in K\in \mathcal{T}_{\leq n}\cap G_n$ there is a cycle $C\subseteq L\in \mathcal{T}_{n+1}$ in $G_n$ with $e\in C.$  To see this, since $G_n$ has only even degree vertices each edge is either in an isolated component (in which case it contains an Euler tour covering this edge) or else has paths from both of its endpoints to $L\setminus V(\mathcal{T}_{\leq n}),$ which is connected by definition.  Applying Lemma \ref{subgraph} to this set and $P$ the endpoints of the paths gives a Borel family of edge disjoint cycles $\mathcal{C}_n$ such that $E(\mathcal{C}_{n+1})\subseteq E(\mathcal{T}_{n+1}),$ and for every minimal $K\in\mathcal{T}_{\leq n}$ with $E(K)\cap G_n\neq \emptyset$ there is a $C\in \mathcal{C}_{n+1}$ with $C\cap K\neq \emptyset.$  Let $G_{n+1}=G_n\setminus E(\mathcal{C}_{n+1}).$

For each cycle $C\in \bigcup_n \mathcal{C}_n$ orient its edges cyclically.  Then each vertex has outdegree equal to indegree in this partial orientation.  Repeating this for each level of the toast, every edge is oriented as at least one edge in a toast piece $K$ is oriented for every piece $L$ with $K\subset L,$ so this gives the desired balanced orientation.
\end{proof}

\subsection{Perfect matchings generically}
\hspace{0.5em}\\

Here we show that every bipartite one-ended $d$-regular Borel graph admits a Borel perfect matching generically.  In the pmp setting the analogous theorem was established in \cite{bowen.kun.sabok} by first finding acyclic fractional perfect matchings and then using the existence of connected toasts a.e. to randomly perturb these matchings, resulting in a new fractional perfect matching with a smaller number of unmatched points.  Since perturbation arguments of this type fundamentally rely on measure, a different approach must be used.  However, the strong pigeonhole principle given to us by the Baire category theorem allows us to immediately construct a perfect matching from a connected toast, bypassing the reduction to the acyclic case in previous arguments and greatly simplifying the proof.

Before beginning, recall that a \textbf{fractional perfect matching} is a function $\sigma:E(G)\rightarrow [0,1]$ such that $\sum_{e: v\in e}\sigma(e)=1$ for every $v\in V(G).$  

\begin{remark}\label{remark:circuit}
    Given a cycle $C\subset E(G)$ with distinguished edge $e'$ and fractional perfect matching $f$ with $\varepsilon\leq f(e)\leq 1-\varepsilon$ for all edges $e\in C,$ we may apply an \textbf{alternating $\varepsilon$-circuit} to obtain a new fractional matching $f',$ where $f'(e)=f(e)+\varepsilon$ if $e\in C$ is of even distance to $e',$ $f'(e)=f(e)-\varepsilon$ if $e$ is of odd distance, and $f'(e)=f(e)$ otherwise.
\end{remark}

\begin{theorem}\label{BP matching}
Let $G$ be a $d$-regular bipartite one-ended Borel graph.  Then $G$ admits a Borel perfect matching generically.
\end{theorem}

\begin{proof}[Proof of Theorem \ref{BP matching}]
Let $\mathcal{T}$ be a connected toast. Let $\sigma_0$ be the (constant) fractional perfect matching that is equal to $\frac{1}{d}$ on all edges.  We inductively define a sequence of Borel fractional perfect matchings $\sigma_i,$ naturals $k_i\in\mathbb{N},$ and Borel families $\mathcal{S}_i$ such that:

\begin{enumerate}
    \item $\sigma_i(e)\in \{\frac{j}{d}: j\in \{0,...,d\}\}$ for all edges.
    \item $\sigma_i(e)\in \{0,1\}$ for all $e\in E(\mathcal{S}_i\cup \partial(\mathcal{S}_i)),$ and $V(\mathcal{S}_i)$ is non-meagre in $U_i$.
    \item $\sigma_i(e)=\frac{1}{d}$ for all $e\in E(G)\setminus E(\mathcal{T}_{k_i}).$
    \item If $\sigma_i(e)\in \{0,1\}$ then $\sigma_{i+1}(e)=\sigma_i(e).$
    \item $G\setminus \mathcal{S}_{<i}$ is one-ended, and there exists a $j$ with $\mathcal{S}_{<i}\subseteq \mathcal{T}_{<j}.$
\end{enumerate}

Given such a sequence, we claim that $M=\bigcup_i \sigma_i^{-1}\{1\}$ is the desired perfect matching.  To see this, note that any vertex is adjacent to at most one edge in $M$ by property (4) of the construction, and that generically many vertices are adjacent to an edge in $M$ by properties (2) and (4).

Now, suppose that we have constructed $\sigma_{i-1},$ $k_{i-1},$ and $\mathcal{S}_{i-1}$ with the desired properties.

\begin{claim}\label{claim_matching}
Let $C_1\in\mathcal{T}_1$, $|E(C_1)|=k$, and let $C_2,...,C_{dk}\in\mathcal{T}$ with $C_j\subsetneq C_{j+1}$ for each $j<k$. Then there is a fractional perfect matching $\sigma$ such that
\begin{enumerate}
    \item $\sigma$ agrees with $\sigma_{i-1}$ outside of $C_k$.
    \item $\sigma(e)\in\set{0,1}$ for each $e\in E(C_1)$.
    \item If $\sigma_{i-1}(e)\in\set{0,1}$, then $\sigma(e)=\sigma_{i-1}(e)$.
\end{enumerate}
\end{claim}

\begin{proof}
First, note that by \cref{remark:circuit}, we may assume without loss of generality that there are no cycles within $C_1$. 

We define $\sigma$ recursively. Once $\sigma(e)$ is integral, remove $e$ from our graph. Note that if we remain a fractional perfect matching at each step, then our graph will never have any leaves.

Pick any remaining (non integral) edge $e\in C_1$, and consider walks along edges leaving both endpoints of $e$. Since $C_1$ has no cycles and since we have no leaves, these walks enter $C_2\setminus C_1$. By connectivity of $C_2\setminus C_1$, these walks may be joined to form a cycle in $C_2$. 
Apply a $\frac{1}{d}$-circuit to this cycle. Note that this may disconnect $C_2\setminus C_1$, but $C_3\setminus C_2$ will still be connected. See \cref{fig:matching} for a partially drawn example with $e$ shown in blue, $d=3$, and all edges $e'$ starting with $\sigma_{i-1}(e')=\frac{1}{3}$. 

\begin{figure}
    \centering
\includegraphics[height=4cm]{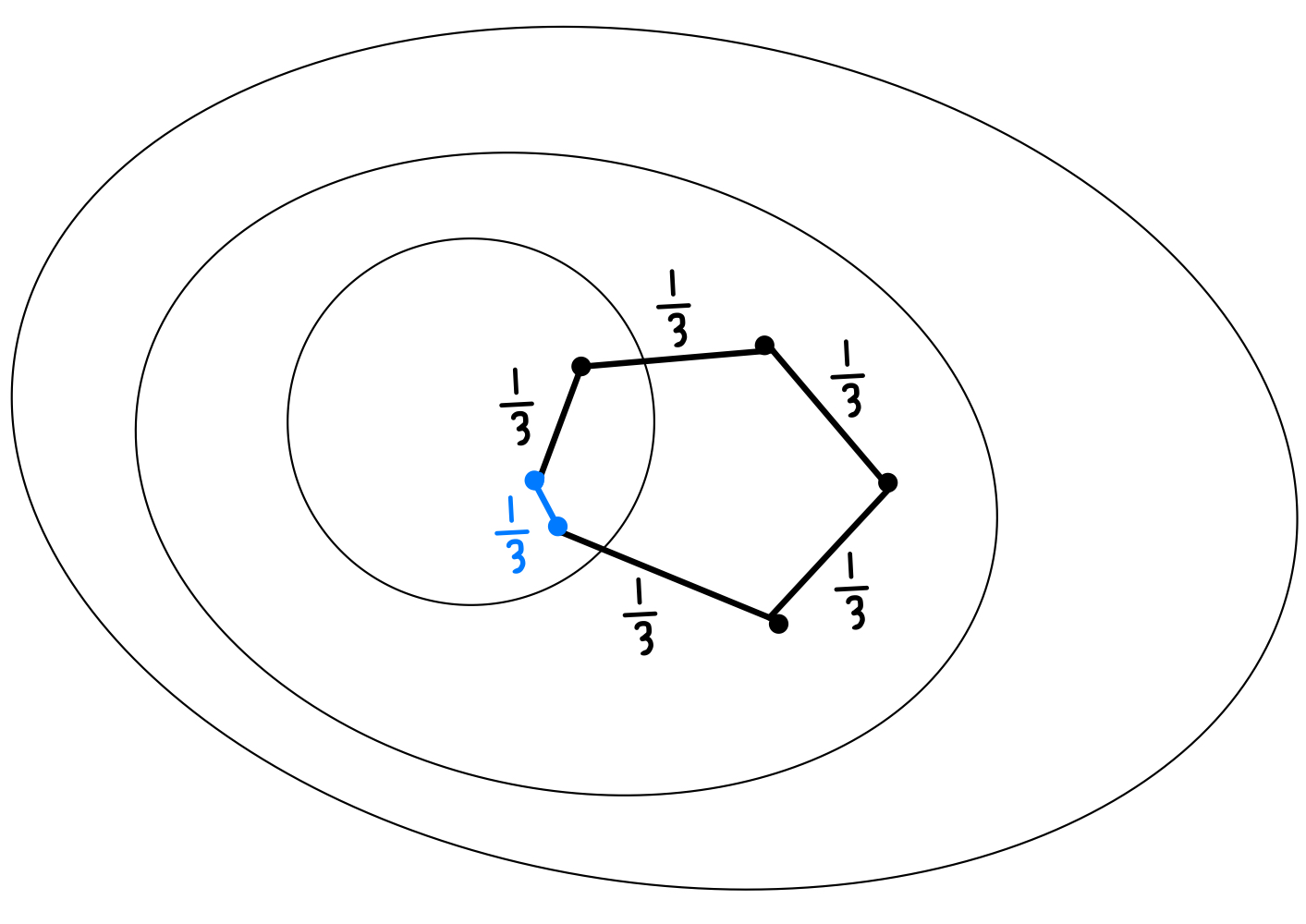}
\includegraphics[height=4cm]{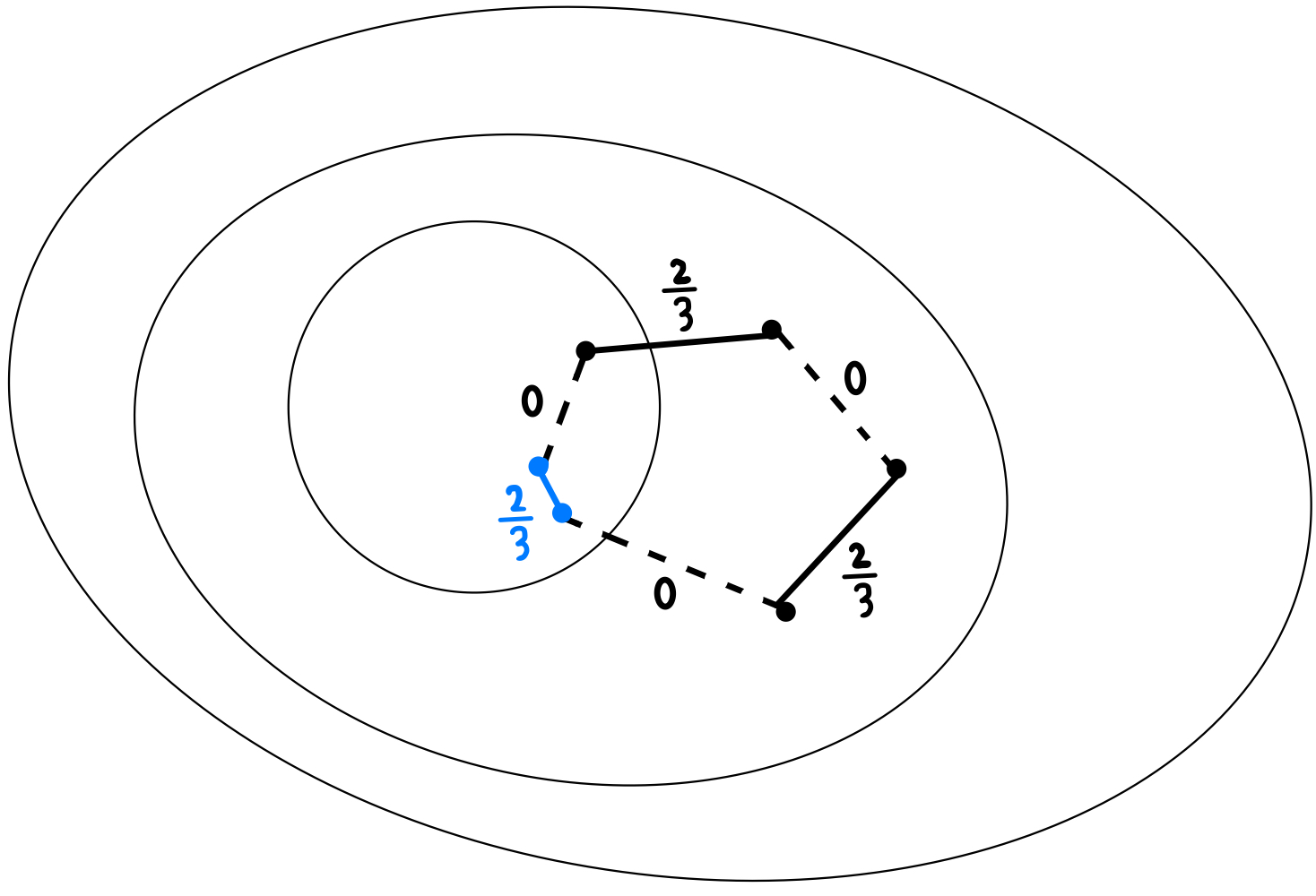}
\includegraphics[height=4cm]{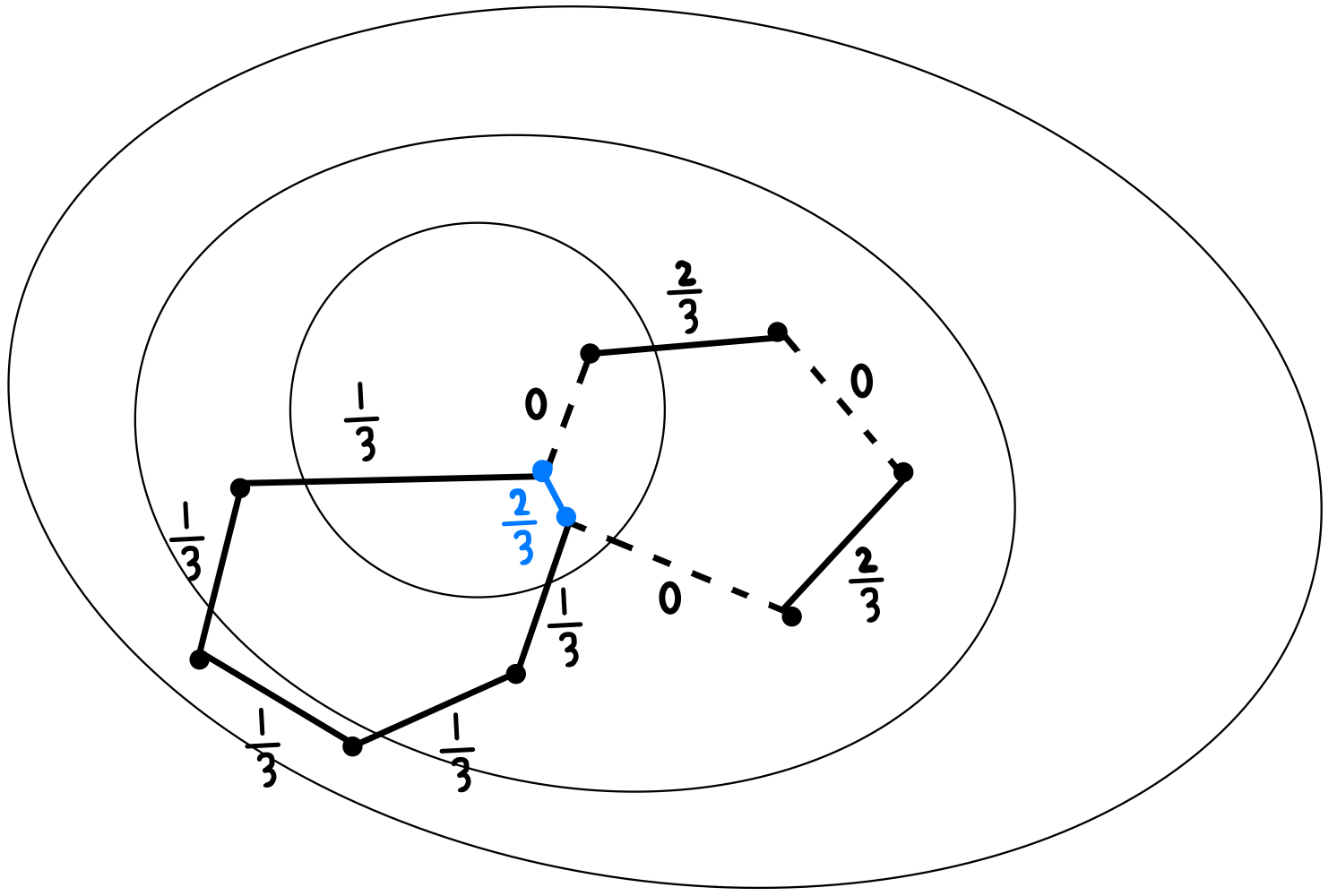}
\includegraphics[height=4cm]{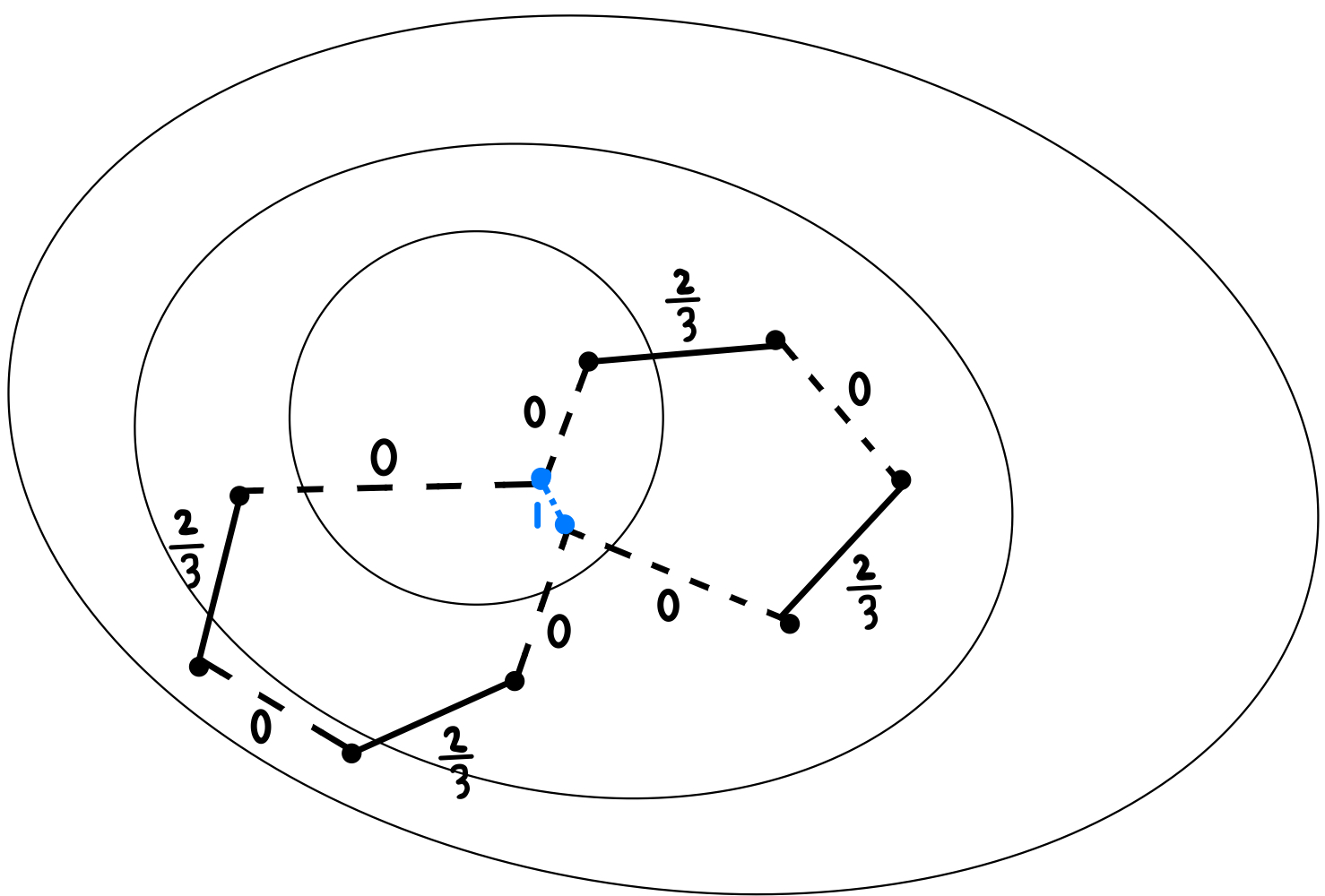}
    \caption{Towards a perfect matching: given a non-integral edge $e\in C_1$ we find a (necessarily even) cycle in $C_2$ and apply a $\frac{1}{d}$ circuit. Iterating this up to $d-1$ many times, we can ensure $\sigma(e)=1$.} 
    \label{fig:matching}
\end{figure}

Iterate this process, considering walks entering $C_3\setminus C_2,...,C_d\setminus C_{d-1}$. In doing this, we can ensure that $\sigma(e)$ becomes integral, and repeating this at most $k$ times ensures that all edges in $C_1$ become integral.
\end{proof}

With this claim in mind, use property (5) of the construction to find a connected toast $\mathcal{T}'\subseteq \mathcal{T}\cap (G\setminus \mathcal{S}_{<i})$ such that $\mathcal{T}'_1$ is non-meagre in $U_i.$ As every element of $\mathcal{T}'_{1}$ is finite, we know that there exists a $t_{i}$ such that $\{v\in C\in \mathcal{T}^{'}_{1}: |C|\leq t_{i}\}$ is non-meagre in $U_i.$  Let $k_i>dt_i$ be large enough such that $\mathcal{T}_i'=\{C\in\mathcal{T}_i: C\subset D_1\subset...\subset D_{k_i} \textnormal{ where each } D_j\in \mathcal{T}_{k_i} \}$ is non-meagre in $U_i.$  Each element of $\mathcal{T}_{k_i}$ contains only finitely many elements of $\mathcal{T}'_{1},$ so we may find a Borel $\omega$-coloring of elements of $\mathcal{T}'_1$ so that no two elements in the same $\mathcal{T}_{k_i}$ tile receive the same color.  Let $\mathcal{S}_i$ be the collection of elements in a color class whose induced vertex set is non-meagre in $U_i.$  

We round edges in $E(\mathcal{S}_i)$ as follows: for each $C\in\mathcal{S}_i$, fix $C\subsetneq D_1\subsetneq...\subsetneq D_{k_i}$, and obtain a fractional perfect matching satisfying the conclusion of \cref{claim_matching}. Since no two elements of $\mathcal{T}_{k_i}$ received the same color, these applications of \cref{claim_matching} do not interfere with each other.
\end{proof}

\begin{remark}
    As in Remark \ref{hcomb tree}, we could use a similar algorithm to construct a perfect matching in any one-ended, bipartite, d-regular highly computable graph.  This is outside the scope of the present paper, but does further the link between Baire measurable and highly computable combinatorics recently explored by Qian and Weilacher in \cite{qian2022descriptive}.
\end{remark}

\subsection{Perfect matchings a.e.}
\hspace{0.5em}\\

 In this section we briefly sketch the following facts:

\begin{theorem}
Let $G$ be a $d$-regular bipartite Borel graph that admits a connected toast.  Then $G$ admits a Borel perfect matching a.e.
\end{theorem}

\begin{cor}
Let $\Gamma$ be a one-ended amenable group with generating set $S$ such that the Cayley graph on $\Gamma$ induced by $S$ is bipartite, and let $G = G(S,a)$ be the Schreier graph of a free Borel action $\Gamma\acts^a X$.  Then $G$ admits a Borel perfect matching a.e. 
\end{cor}

This follows essentially from Theorem \ref{toast a.e.} and the arguments in \cite{bowen.kun.sabok}, where the same theorem was proven for pmp graphs.  Other than to show the existence of connected toasts, the only place that the pmp assumption was used in Theorem 6.4 of \cite{bowen.kun.sabok} (the pmp version of the result we sketch here) was in order to ensure that extreme points in the space of measurable fractional perfect matchings are integral besides on a union of bi-infinite lines.  This is not necessarily true for non-pmp graphs, but we can circumvent this issue with the following lemma.

\begin{lemma}
Let $\sigma$ be a measurable fractional perfect matching such that $\sigma^{-1}(0,1)$ is a.e. acyclic.  Then there is a measurable fractional perfect matching $\sigma'$ so that $\sigma'^{-1}(0,1)\subseteq \sigma^{-1}(0,1)$ and $\sigma'$ is integral besides on a union of bi-infinite lines.
\end{lemma}

\begin{proof}
Since the subgraph induced by $\sigma^{-1}(0,1)$ has no degree $1$ vertices and is hyperfinite, we may apply \cite{conley.miller.pm} Theorem A to find a measurable matching $M$ of this graph which covers all but a union of bi-infinite lines $B$.  Letting $\sigma'(e)=\sigma(e)$ for all $e\notin \sigma^{-1}(0,1),$ $\sigma'(e)=1$ for all $e\in M,$ $\sigma'(e)=\frac{1}{2}$ for $e\in B,$ and $\sigma'(e)=0$ otherwise is as desired.
\end{proof}

\bibliographystyle{amsalpha} 
\bibliography{1} 
\end{document}